%% file: extent-sampling.tex
\documentclass[]{siamart171218}
\usepackage{algpseudocode}
\usepackage{amsfonts}
\usepackage{amsmath}
\usepackage{amssymb}
\usepackage{commath}
\usepackage{cleveref}
\usepackage{subfig}
\usepackage{tikz}
\usetikzlibrary{arrows.meta,datavisualization}

\DeclareMathOperator{\expectation}{\mathbb{E}}
\newcommand{\E}[1]{\expectation \sbr{#1}}

\DeclareMathOperator{\variance}{Var}
\newcommand{\var}[1]{\variance \sbr{#1}}

\newcommand{\volumefrontfactor}{\frac{\pi^{\frac{n}{2}}}{\Gamma \del{\frac{n}{2} + 1}}}
\newcommand{\surfaceareafrontfactor}{\frac{n\pi^{\frac{n}{2}}}{\Gamma \del{\frac{n}{2} + 1}}}
\newcommand{\error}{\tilde{\varepsilon}}
\newcommand{\bigo}[1]{\ensuremath{\mathcal{O}\del{#1}}}
\newcommand{\sphereset}{\ensuremath{S^{n-1}}}

\newcommand{\uniformdist}{$\mathrm{uniform}(0,1)$ }

\newcommand{\betadist}{$\mathrm{beta}(\alpha=2,\beta=2)$  }

\title{An algorithm for estimating volumes and other integrals in $\lowercase{n}$ dimensions}
\author{Arun I. \and Murugesan Venkatapathi \thanks{Department of Computational and Data
    Sciences, Indian Institute of Science, Bengaluru - 560012
    (\email{murugesh@iisc.ac.in})}}

\input{data/volume-body-parameters}
\input{data/volume-distribution-parameters}
\input{data/integral-parameters}

\begin{document}

\maketitle

\begin{abstract}
  The computational cost in evaluation of the volume of a body using
  numerical integration grows exponentially with dimension of the
  space $n$. The most generally applicable algorithms for
  estimating $n$-volumes and integrals are based on Markov Chain Monte
  Carlo (MCMC) methods, and they are suited for convex domains. We analyze a less known alternate method used for estimating $n$-dimensional
  volumes, that is agnostic to the convexity and roughness of the body. It results due to the possible decomposition of an arbitrary $n$-volume into an integral of statistically weighted
  volumes of $n$-spheres. We establish its dimensional scaling, and extend it for evaluation of arbitrary integrals over non-convex domains. Our results also show that this method is significantly more efficient than the MCMC approach even when restricted to convex domains, for $n$ $\lesssim$ 100. An importance sampling may extend this advantage to larger dimensions.
\end{abstract}

\section{Introduction}

Analytic evaluation of volumes is feasible for a relatively small set
of symmetric bodies defined in the appropriate coordinate systems. In
some cases, the surface of a body may not have a tractable closed form
analytical expression and the body may only be defined by a set of
inequalities. These challenges in analytical integration were overcome
by numerical methods \cite{keshavarzzadeh2018numerical}. As the
dimension of problems became large, the exponential increase in the
cost of numerical methods (NP-hardness) inspired new statistical
methods that converge to a reasonable estimate of the volume in
polynomial time under certain constraints
\cite{allgower1986volume,speevak1986volume,lawrence1991volume}. In
evaluating more general integrals, deterministic sampling methods such
as the Quasi-Monte Carlo are very efficient when the integrand can be
reduced to a function of a single effective variable
\cite{dick2007qmc}. Similarly, the naive Monte Carlo method is largely
effective when the limits of the integration are constants, that is,
over a domain which is an $n$-orthotope (a rectangle when $n=2$,
cuboid when $n=3$ etc.). In problems where some function defines the
boundary of the domain or its membership, and in problems where the
sampled independent variables have an implicit non-uniform probability
density, correctly sampling the domain in itself amounts to be
NP-hard.

Even for the diminished problem of estimating
$n$-volumes, a Markov Chain Monte Carlo (MCMC) sampling is the only
tractable approach for large $n$ \cite{simonovits2003compute}. This
approach is geometrically insightful and involves cancellation of
errors in the estimates, resulting in relatively fast convergence for
convex volumes. Nevertheless, after improving rapidly from
$\bigo{n^{23}}$ scaling in the samples required
\cite{dyer1988complexity,dyer1991computing,kannan1997random},
algorithms using this approach have stagnated at $\bigo{n^4}$ samples
for a given convex shape
\cite{lovasz2003simulated,jaekel2011monte,ge2015fast}. Since the cost
of evaluating a typical scalar function increases linearly with the
number of cardinal directions $n$, the total computing effort in
estimating volumes scales as $\bigo{n^5}$ for these MCMC methods. This
general poor scaling of the MCMC approach with the dimension is
overcome using specialized algorithms designed for certain forward and
inverse problems \cite{chen2016scaling, constantine2016scaling,
  feng2018scaling, vollmer2015scaling}. Volumes of non-convex bodies can also be evaluated more accurately using semi-definite programming, but they are suited for smaller dimensions \cite{henrion2009approximate}.

The algorithm presented here is suitable for estimating volumes of both
convex and non-convex bodies with fewer exceptions, and for other
problems of estimation in continuous spaces. This method also retains
the advantages of the naive Monte Carlo sampling such as the full
independence of the random samples. The resulting suitability for
parallel computation could be of additional significance. The proposed
$n$-sphere-Monte-Carlo (NSMC) method decomposes the estimated volume
into weighted volumes of $n$-spheres, and these weights are trivially
estimated by sampling extents of the domain with respect to an
origin. Such a volume preserving transformation was suggested many years ago \cite{fok1989volume}. We also show a straightforward adaptation of this method to
estimate arbitrary integrals. Here, the required number of extent
samples scale as $\bigo{n}$ for a fixed distribution of extents of the
domain, with the corresponding total computing effort scaling as
$\bigo{n^2}$ for estimating volumes and as $\bigo{n^3}$ for estimating
arbitrary integrals. While estimating volumes using this approach involves only sampling the extents, estimating arbitrary integrals includes sampling the interior of the domain. The proposed approach may
have challenges in estimating volumes which are not just highly
eccentric but also have a tailed distribution of large extents, such as certain convex $shapes$. In
such cases, the poor scaling in number
of samples with $n$ can be reduced by an appropriate importance
sampling to capture the tailed extents. The challenges in such sampling of high dimensional
sub-spaces along with a potential solution has been described
elsewhere \cite{arun2021thesis, arun2021cones}. In this paper, we limit ourselves to the naive NSMC
approach using an unbiased sampling of the extents. The naive algorithm is significantly more efficient than the MCMC approach even when restricted to convex domains, for $n$ $\lesssim$ 100.

\section{Frequently used terms and symbols}

\begin{description}
\item[\sphereset] is the set of all points on the surface of the unit
  sphere.
\item[membership function] It is a
  function $\mathbb{R}^n \mapsto \cbr{0, 1}$ that maps a point in
  space, to 0 if that point lies outside the body, or 1 if that point
  lies inside the body.
\item[extent] The extent of a body is the distance
  between the origin of the coordinate system and a point on the
  surface of the body. If $r$ is the extent of a
  body along the direction vector $\hat{s}$, then $r \hat{s}$ lies on
  the surface of the body.
\item[extent function] The extent function of a body is a function
  $\sphereset \mapsto \mathbb{R}$ that maps a direction vector to a
  corresponding extent of the body.
\item[extent density] The extent density of a body is the probability
  density function of extents obtained when direction vectors are
  randomly sampled from a uniform distribution on $\sphereset$.
\item[$\mathbf{s_n}$] is the surface area of the $n$ dimensional unit
  sphere given by
  \begin{equation}
    s_n = \surfaceareafrontfactor
  \end{equation}
\item[$\mathbf{v_n}$] is the volume of the $n$ dimensional unit sphere
  given by
  \begin{equation}
    v_n = \frac{s_n}{n} = \volumefrontfactor
  \end{equation}
\end{description}

\section{Estimation of volume}
\subsection{Problem statement}

Given a closed body containing the origin, specified by an extent
function $S$ with an extent density function $f_R$, estimate the volume
enclosed by the body.

We assume $S$ to be single valued for clarity of the paper but it need
not be continuous. The constraint of $S$ being single valued leaves
out some non-convex geometries such as in
\cref{fig:multi-valued}. This constraint can be relaxed by a simple
generalization of the extent of such a body as shown in
\cref{sec:multi-valued-extent-function}. Also, in many cases, the
extent function may not be available explicitly and only a membership
function may be available. In such cases, we can construct an extent
function that estimates the extent in the given direction by
repeatedly invoking the membership function for points along that
direction, say using a bisection search.

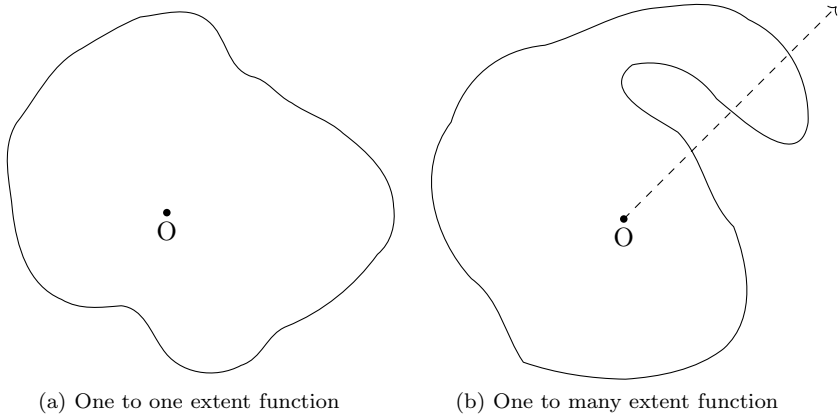
\begin{figure}
  \centering
  \subfloat[One to one extent function] {
    \label{fig:single-valued}
    \begin{tikzpicture}[scale=0.5]
      \draw (-2.7967342,-2.3106736)
      .. controls (-3.7043739,-1.9346336) and (-3.9904398,-0.87996361) .. (-4.0992726,0.00705339)
      .. controls (-4.1636362,0.79238339) and (-4.4375277,1.6640734) .. (-3.994924,2.3894634)
      .. controls (-3.4512236,3.0727934) and (-3.0812982,3.9468534) .. (-2.2660806,4.3646434)
      .. controls (-1.7634807,4.6707534) and (-1.250036,4.9878634) .. (-0.7071386,5.2049534)
      .. controls (-0.1135082,5.2717734) and (0.5852213,5.5301534) .. (1.1007503,5.0914834)
      .. controls (1.6043903,4.6953334) and (1.5825563,3.8394334) .. (2.2514523,3.6218834)
      .. controls (2.6797313,3.5469634) and (2.9295208,3.1081434) .. (3.3196748,2.9210134)
      .. controls (3.7572768,2.6129334) and (4.2916648,2.4781834) .. (4.6873848,2.1087234)
      .. controls (5.3158958,1.6291134) and (5.9969958,1.0138734) .. (6.0284188,0.16540339)
      .. controls (6.0956498,-0.29785661) and (5.9710468,-0.80214361) .. (5.6006978,-1.1088136)
      .. controls (4.9872868,-1.9550536) and (4.1515928,-2.6386536) .. (3.1797558,-3.0291936)
      .. controls (2.6520643,-3.2376136) and (2.5369843,-3.8908536) .. (1.9763153,-4.0630036)
      .. controls (1.2860483,-4.4381036) and (0.3109613,-4.2998436) .. (-0.1252936,-3.6072936)
      .. controls (-0.4264775,-3.1922636) and (-0.622766,-2.5394636) .. (-1.214215,-2.4780936)
      .. controls (-1.7438697,-2.5217436) and (-2.3149993,-2.6011236) .. (-2.7967342,-2.3106736);
      \fill [black] (0,0) circle [radius=0.1];
      \node [below] at (0,0) {O};
    \end{tikzpicture}
  }
  \subfloat[One to many extent function] {
    \label{fig:multi-valued}
    \begin{tikzpicture}[scale=0.5]
      \draw (0.78242456,5.6090091)
      .. controls (1.6379134,5.6871563) and (2.5727205,5.8975436) .. (3.3137353,5.3062475)
      .. controls (4.3806108,4.8214915) and (4.9431878,3.7082418) .. (4.9005645,2.5685865)
      .. controls (4.6758216,1.1507455) and (3.0289442,2.7430381) .. (2.4646791,3.2019168)
      .. controls (1.9496533,3.9246509) and (1.1027784,4.295187) .. (0.22643958,4.1007947)
      .. controls (-0.7008227,3.3836869) and (0.87916567,2.6809839) .. (1.4371031,2.3093844)
      .. controls (2.1392545,1.5890409) and (2.204946,0.49910621) .. (2.9210374,-0.20450456)
      .. controls (3.3019404,-1.2236986) and (3.5663871,-2.6237937) .. (2.6563396,-3.4498146)
      .. controls (1.9300553,-4.0170542) and (0.95034707,-4.1980002) .. (0.04536859,-4.2628156)
      .. controls (-0.87775328,-4.2466093) and (-1.7936578,-4.1123965) .. (-2.670961,-3.8104695)
      .. controls (-3.1644689,-3.087733) and (-3.2896594,-2.150226) .. (-4.0504705,-1.5880672)
      .. controls (-5.0569542,-0.49124193) and (-5.5643706,1.2978843) .. (-4.5901254,2.5788485)
      .. controls (-4.2356465,3.7504988) and (-3.3117539,4.5117254) .. (-2.0968061,4.6235095)
      .. controls (-1.1129049,4.8975855) and (-0.24345343,5.4894729) .. (0.78242456,5.6090091);
      \fill [black] (0,0) circle [radius=0.1];
      \node [below] at (0,0) {O};
      \draw [dashed,->] (0,0) -- (45:8);
    \end{tikzpicture}
  }
  \caption{Two bodies --- the first with a unique extent in every
    direction, i.e., a one to one extent function and the second with
    multiple extents in some directions, i.e., a one to many extent
    function. In \cref{fig:multi-valued}, a representative direction
    in which the extent function is multivalued is indicated by an
    arrow. Point O is the origin from which extents are measured.}
\end{figure}

\subsection{Solution}

We repose the problem of $n$-dimensional integration for volume in
spherical coordinates, as an estimation of the relative weights for
the volumes of spheres of varying radii that add up to the volume of
the given body. This approach allows a simple statistical estimation
of the volume of even arbitrary non-convex bodies and requires no
lower bounds on the smoothness of the body. The two dimensional
illustration in \cref{fig:2d-motivation} serves as a simple example.

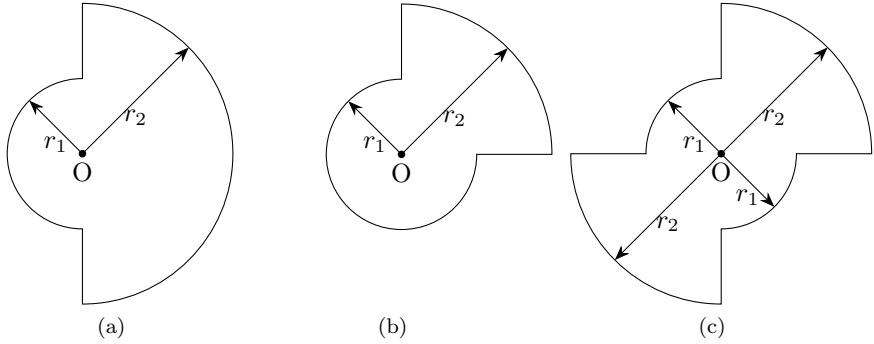
\begin{figure}
  \centering
  \subfloat[] {
    \label{fig:two-equal-sectors}
    \begin{tikzpicture}
      \draw (0,-2) arc [start angle=-90, delta angle=180, radius=2]
      -- ++(0,-1)
      arc [start angle=90, delta angle=180, radius=1]
      -- cycle;
      \draw [-{Stealth[length=2mm]}] (0,0) -- node[below] {$r_1$} (135:1);
      \draw [-{Stealth[length=2mm]}] (0,0) -- node[below] {$r_2$} (45:2);
      \fill [black] (0,0) circle [radius=0.05];
      \node [below] at (0,0) {O};
    \end{tikzpicture}
  }
  \subfloat[] {
    \label{fig:two-unequal-sectors}
    \begin{tikzpicture}
      \draw (2,0) arc [start angle=0, delta angle=90, radius=2]
      -- ++(0,-1)
      arc [start angle=90, delta angle=270, radius=1]
      -- cycle;
      \draw [-{Stealth[length=2mm]}] (0,0) -- node[below] {$r_1$} (135:1);
      \draw [-{Stealth[length=2mm]}] (0,0) -- node[below] {$r_2$} (45:2);
      \path (-2,0) -- (2,0);
      \path (0,-2) -- (0,2);
      \fill [black] (0,0) circle [radius=0.05];
      \node [below] at (0,0) {O};
    \end{tikzpicture}
  }
  \subfloat[] {
    \label{fig:four-equal-sectors}
    \begin{tikzpicture}
      \draw (2,0) arc [start angle=0, delta angle=90, radius=2]
      -- ++(0,-1)
      arc [start angle=90, delta angle=90, radius=1]
      -- ++(-1,0)
      arc [start angle=180, delta angle=90, radius=2]
      -- ++(0,1)
      arc [start angle=270, delta angle=90, radius=1]
      -- cycle;
      \draw [-{Stealth[length=2mm]}] (0,0) -- node[below] {$r_1$} (135:1);
      \draw [-{Stealth[length=2mm]}] (0,0) -- node[below] {$r_1$} (-45:1);
      \draw [-{Stealth[length=2mm]}] (0,0) -- node[below] {$r_2$} (45:2);
      \draw [-{Stealth[length=2mm]}] (0,0) -- node[below] {$r_2$} (-135:2);
      \fill [black] (0,0) circle [radius=0.05];
      \node [below] at (0,0) {O};
    \end{tikzpicture}
  }
  \caption{Consider the 2 dimensional body in
    \cref{fig:two-equal-sectors} consisting of two semicircles of
    radius $r_1$ and $r_2$ attached to each other. The 2-volume (or
    area) of this composite body is
    $\pi(\frac{1}{2} r_1^2 + \frac{1}{2} r_2^2)$. Likewise, the
    2-volume of the body in \cref{fig:two-unequal-sectors} is
    $\pi(\frac{3}{4} r_1^2 + \frac{1}{4} r_2^2)$. We may observe that
    given extents R in all directions, the 2-volume of an arbitrary
    body is simply the mean of $R^2$ with a multiplying front constant
    $\pi$; this front constant depends on the dimension of space. Note
    that the angular sectors with identical radii need not be
    contiguous, and even the body of \cref{fig:four-equal-sectors} has
    the same 2-volume as that of \cref{fig:two-equal-sectors}.}
  \label{fig:2d-motivation}
\end{figure}

The volume of a body in spherical coordinates, with $\rho$ being the
radial coordinate and $\dif \hat{s}$ being the surface element of the
unit sphere, is
\begin{equation} \label{eq:volume-in-spherical-coordinates}
  V = \oint_{\sphereset}\int_0^{r=S(\hat{s})} \rho^{n-1}\dif \rho \dif \hat{s} = \frac{1}{n} \oint_{\sphereset} r^n \dif \hat{s}
\end{equation}
While the above form is convenient for analytic integration when $S$
is tractable and known, it is best avoided otherwise. But this form is
well suited for a statistical estimation by uniform sampling on the
surface of the unit sphere as given below.

If $R$ is a random variable representing the extent obtained when
sampling direction vectors uniformly distributed on $\sphereset$,
\cref{eq:volume-in-spherical-coordinates} can be rewritten using the
expectation of $R^n$ in terms of the extent density $f_R$ of the body.
\begin{equation}
  V = \cbr{\frac{1}{n} \oint_{\sphereset} \dif \hat{s}} \cbr{\int_0^\infty r^n f_R(r) \dif r}
\end{equation}
Expressing in terms of the surface area $s_n$ and volume $v_n$ of the
unit sphere,
\begin{equation}
  \label{eq:algorithm-equation}
  V = \frac{s_n}{n}\int_0^\infty r^n f_R(r) \dif r = v_n \E{R^n}
\end{equation}

For the purpose of volume estimation, classifying bodies based on
their extent densities is more convenient especially for non-convex
and non-symmetric bodies.

If the extent density $f_R$ is known, one can integrate
\cref{eq:algorithm-equation} using a numerical quadrature, and this
scales only as $\bigo{n}$ in the total computing effort. But in
practice, for an unknown body, the estimation and the integration of
the extent density are implemented as a single algorithm represented
by \cref{eq:algorithm-equation} and shown in \cref{alg:volume}.

\begin{algorithm}
  \caption{Estimate volume}
  \begin{algorithmic}
    \Procedure{Estimate volume}{$S$}
    \State $V \gets 0$
    \For{$i = 1:N$}
    \State $\hat{s_i} \gets \text{unit vector in random direction}$
    \State $R_i \gets S(\hat{s_i})$
    \State $V \gets V + R_i^n$
    \EndFor
    \State $V \gets \frac{v_n}{N} V$
    \State \Return $V$
    \EndProcedure
  \end{algorithmic}
  \label{alg:volume}
\end{algorithm}

There are two significant advantages to this statistical estimation.
\begin{enumerate}
\item For a body with a given extent density, the number of random
  samples required for the convergence of the $n^{th}$ moment of the
  extent density, i.e., the $n$-dimensional volume of the body, has an
  upper-bound that varies as $\bigo{n}$. This is proved in
  \cref{sn:analysis}.
\item The independence of the random samples is maintained, and hence
  it is suitable for parallel computing approaches.
\end{enumerate}

The simplest extent density is $f_R(r) = \delta(r-r_0)$ for a sphere
of radius $r_0$. Some convex bodies, such as the cube, are well
defined by their symmetries for all dimensions, while their extent
densities change with dimension. Conversely, different bodies,
including their different orientations, can result in the same extent
density. Different reference points or origins can result in different extent densities for the same body, and thus
affect the convergence weakly but not the order of convergence with
$n$. Also, note that iterating the point of reference to the nominal centre of the body requires only $\bigo{n}$ extent samples, in any case. Further analysis of this algorithm and numerical results for
demonstration follow in the later sections.

\section{Analysis}
\label{sn:analysis}

Approximating the expectation in \cref{eq:algorithm-equation} using a
Monte Carlo estimate $V_N$ of $N$ samples,
\begin{equation}
  V_N = \frac{v_n}{N} \sum_{i=1}^N R_i^n
\end{equation}
The expected root-mean-square (RMS) error of this estimate can then be
written as
\begin{equation}
  \varepsilon = \sqrt{\var{V_N}} = v_n \sqrt{\var{R^n}} \frac{1}{\sqrt{N}}
\end{equation}
We normalize this RMS error with the true volume from
\cref{eq:algorithm-equation} to obtain the relative error $\error$.
\begin{equation}
  \label{eq:normalized-error}
  \error = \frac{\varepsilon}{V} = \frac{\sqrt{\var{R^n}}}{\E{R^n}}
\frac{1}{\sqrt{N}}
\end{equation}

We then pose the analysis of the relative error as derivation of a
bounds for the variance-to-square-mean ratio of the $n^{th}$ moment of
a random variable in a Hausdorff moment problem. Using the above, we
establish the scaling of the number of samples $N$ for any given
relative RMS error $\error$ in terms of the number of dimensions $n$
in the volume estimation.

\subsection{Scale invariance of relative error in the volume estimate}
\label{ssn:scale-invariance-of-relative-error}

Suppose extents of an $n$-dimensional body were scaled by a factor
$a$,
\begin{equation}
  \error = \frac{\sqrt{\var{(aR)^n}}}{\E{(aR)^n}} \frac{1}{\sqrt{N}} =
  \frac{a^n \sqrt{\var{R^n}}}{a^n \E{R^n}} \frac{1}{\sqrt{N}} =
\frac{\sqrt{\var{R^n}}}{\E{R^n}} \frac{1}{\sqrt{N}}
\end{equation}
Thus, relative error in the volume estimate is invariant under a
scaling of the body. Without any loss of generality of our analysis,
it is sufficient to only consider bodies with extents ranging from
$\frac{1}{\lambda}$ to $1$, where $\lambda > 1$ represents the ratio
of the largest to the smallest extent of the body. Hence, in our
analysis, we only consider extent densities with compact support
$\intcc{\frac{1}{\lambda},1}$. Likewise, the convergence of the
algorithm itself is not affected by the scale of the body; only the
distribution of relative extents matters.

\subsection{Scaling of relative error with dimension}

Given that the extent density of interest has been reduced to a
compact support $\intcc{\frac{1}{\lambda},1}$, we have the following
theorems on moments of $R$ and their variance-to-square-mean ratio. We
consider boundaries given by a continuous extent function $S$, where
the extent density is also continuous, bounded and greater than zero
in the interval $\intcc{\frac{1}{\lambda},1}$. If the boundary is
defined by a function $S$ that is not continuous, all the possible
relative extents in $\intcc{\frac{1}{\lambda},1}$ need not exist and
the extent density can indeed be discontinuous or zero at points
within the interval. The following theorems nevertheless apply to such
extent densities in a piece-wise manner with rescaling, thus we incur
no loss of generality in the bodies considered.

\begin{lemma}
  \label{th:bounds-on-moments}
  If $X$ is a random variable whose probability density function $f$
  is supported on $\intcc{\frac{1}{\lambda},1}$ where
  $\lambda \in \intoo{1, \infty}$, and $f$ is bounded as
  $f_{max} \geq f(x) \geq f_{min} > 0$ for all
  $x \in \intcc{\frac{1}{\lambda}, 1}$, then for all
  $k \in \mathbb{N}$
  \[ \frac{f_{max}}{k+1} \del{1-\frac{1}{\lambda^{k+1}}} \ge \E{X^k} \ge \frac{f_{min}}{k+1} \del{1-\frac{1}{\lambda^{k+1}}} \]
\end{lemma}

\begin{proof}
  When $f$ is bounded as $f_{max} \geq f(x) \geq f_{min} > 0$ for all
  $x \in \intcc{\frac{1}{\lambda}, 1}$, its moments can be trivially
  bounded by zeroth order approximations as given below.
  \begin{equation}
    \int_\frac{1}{\lambda}^1  x^k f_{max} \dif x \geq \int_\frac{1}{\lambda}^1  x^k f(x) \dif x \geq \int_\frac{1}{\lambda}^1 x^k f_{min} \dif x
  \end{equation}
  resulting in
  \begin{equation}
    \frac{f_{max}}{k+1} \del{1-\frac{1}{\lambda^{k+1}}} \ge \E{X^k} \ge \frac{f_{min}}{k+1} \del{1-\frac{1}{\lambda^{k+1}}}
  \end{equation}
\end{proof}

Using bounds of \cref{th:bounds-on-moments}, we can now establish that
the variance-to-square-mean ratio relating the number of samples and
corresponding error, varies as $\bigo{n}$ for a volume in $n$
dimensions.

\begin{theorem}
  \label{th:upper-bound-on-standard-deviation-to-mean}
  If $X$ is a random variable whose probability density function $f$
  is supported on $\intcc{\frac{1}{\lambda},1}$ where
  $\lambda \in \intoo{1, \infty}$, and $f$ is bounded as
  $f_{max} \geq f(x) \geq f_{min} > 0$ for all
  $x \in \intcc{\frac{1}{\lambda}, 1}$, then, for
  $k \gg \frac{1}{\lambda -1}$ and $k \in \mathbb{N}$, there exists
  some $c \in \mathbb{R}$ such that
  \[ \frac{\sqrt{\var{X^k}}}{\E{X^k}} \leq c \sqrt{k} \]
\end{theorem}

\begin{proof}
  \begin{equation}
    \frac{\var{X^k}}{\cbr{\E{X^k}}^2} = \frac{\E{X^{2k}} - \cbr{\E{X^k}}^2}{\cbr{\E{X^k}}^2}
  \end{equation}
  \begin{equation}
    \label{eq:variance-to-square-mean-ratio}
    \frac{\var{X^k}}{\cbr{\E{X^k}}^2} = \frac{\E{X^{2k}}}{\del{\E{X^k}}^2} - 1
  \end{equation}
  The two bounds on moments in \cref{th:bounds-on-moments} applied to
  maximize the above ratio, gives us
  \begin{equation}
    \frac{\var{X^k}}{\cbr{\E{X^k}}^2} \leq \frac{f_{max}}{f_{min}^2} \frac{(k+1)^2}{2k+1} \frac{(1-\frac{1}{\lambda^{2k+1}})}{(1-\frac{1}{\lambda^{k+1}})^2} - 1
  \end{equation}
  and further for all $k \gg \frac{1}{\lambda-1}$,
  \begin{equation}
    \frac{\sqrt{\var{X^k}}}{\E{X^k}} \leq c \sqrt{k}
  \end{equation}
  where \[ c \approx \frac{\sqrt{f_{max}}}{f_{min}} \]
\end{proof}

In the proposition below, we use the variance-to-square-mean ratio of
the $k^{\text{th}}$ moment varying as $\bigo{k}$, to derive the
expected number of samples for a given error.
\begin{proposition}
  \label{th:scaling-relative-error-with-n}
  For a given extent density and a relative RMS error $\error$ in the
  volume estimate $V_N$, the required number of samples $N$ increases
  linearly with dimension $n$.
\end{proposition}

\begin{proof}
  Applying the bound on the variance-to-square-mean ratio in
  \cref{th:upper-bound-on-standard-deviation-to-mean} into
  \cref{eq:normalized-error} for relative RMS error, gives us the
  following.
  \begin{equation}
    \error \le \frac{\sqrt{f_{max}}}{f_{min}} \frac{\sqrt{n}}{\sqrt{N}}
  \end{equation}
  \begin{equation}
    N \le \frac{f_{max}}{f_{min}^2} \frac{n}{\error^2}
  \end{equation}
\end{proof}

It can be shown that the number of samples required for a relative RMS
error $\error$ is significantly smaller than this upper bound for
probability density functions supported on $\intcc{0,1}$ that do not
have a tail along large extents. Exact relations for moments of a few
distributions are shown in the following corollary, and other demonstrations are shown in \cref{sec:numerical-results}. For extent
densities $f_R$ that are tailed along large extents, an importance
sampling can limit the number of samples to a reasonable value.

\begin{corollary}
  \label{th:corollary-on-moments}
  We present a few distributions where exact analytical relations for
  the variance-to-square-mean ratio of the $k^{\text{th}}$ moment can
  be derived.
\end{corollary}

\begin{proof}
  For the uniform distribution on the interval $\intcc{0,1}$,
  \begin{equation}
    f(x) =
    \begin{cases}
      1 & x \in \intcc{0, 1} \\
      0 & \text{elsewhere}
    \end{cases}
  \end{equation}
  \begin{equation}
    \frac{\sqrt{\var{X^k}}}{\E{X^k}} = \frac{k}{\sqrt{2k+1}}
  \end{equation}

  For the distribution with the polynomial probability density
  function given by the following with $m \ne -1$.
  \begin{equation}
    \label{eq:polynomial-distribution-pdf}
    f(x) =
    \begin{cases}
      (m + 1) x^m & x \in \intcc{0, 1} \\
      0 & \text{otherwise}
    \end{cases}
  \end{equation}
  \begin{equation}
    \frac{\sqrt{\var{X^k}}}{\E{X^k}} = \frac{k}{\sqrt{(m+1)(m+2k+1)}}
  \end{equation}

  For a U-quadratic distribution with a probability density function
  given by
  \begin{equation}
    \label{eq:u-quadratic-distribution-pdf}
    f(x) =
    \begin{cases}
      12\del{x - \frac{1}{2}}^2 & x \in \intcc{0, 1} \\
      0 & \text{otherwise}
    \end{cases}
  \end{equation}
  \begin{equation}
    \frac{\sqrt{\var{X^k}}}{\E{X^k}} = \sqrt{\frac{(2k^2 + k + 1)(k +
        1)(k + 2)^2(k + 3)^2}{3(2k + 1)(2k + 3)(k^2 + k + 2)^2} - 1}
  \end{equation}
  For large $k$,
  \begin{equation}
    \frac{\sqrt{\var{X^k}}}{\E{X^k}} \approx \sqrt{\frac{k}{6}}
  \end{equation}
\end{proof}

\section{Estimation of arbitrary integrals}

In this section we extend the proposed algorithm to estimate arbitrary
integrals.

\subsection{Problem statement}

Given a function $h$ defined over an arbitrary domain specified by an
extent function $S$, estimate the integral of $h$ over the domain.

\subsection{Solution}

The required integral in spherical coordinates, with $\rho$ being the
radial coordinate and $\dif \hat{s}$ being the surface element of the
unit sphere, is
\begin{equation}
  I = \oint_{\sphereset} \int_0^{S(\hat{s})} \rho^{n-1} h(\rho\hat{s}) \dif \rho \dif \hat{s}
\end{equation}
Let $i(\hat{s})$ be the integral along $\rho$ for a given $\hat{s}$.
\begin{equation}
  i(\hat{s}) = \int_0^{S(\hat{s})} \rho^{n-1} h(\rho\hat{s}) \dif \rho
\end{equation}
Then, the integral $I$ over the arbitrary domain is
\begin{equation}
  I = \int_{\sphereset} i(\hat{s}) \dif \hat{s} = s_n \E{i(\hat{s})}
\end{equation}

An algorithm implementing this expectation is shown in
\cref{alg:integral}. Note that any importance sampling applied to
$R^n$ in estimating volumes, can also be extended to $i(\hat{s})$
in the problem of integration over a domain. In this work, we present
results of a hybrid approach to the problem of $n$ dimensional
integration, where one dimensional integration of $i(\hat{s})$ along
the radial direction is performed using deterministic quadrature
schemes such as Gaussian quadrature, while the high dimensional
partial integral over the angular coordinates is estimated
statistically using the naive NSMC approach. Alternative approaches
for integration using the naive NSMC are possible.

\begin{algorithm}
  \caption{Estimate arbitrary integral}
  \begin{algorithmic}
    \Procedure{Estimate arbitrary integral}{$i$}
    \State $I \gets 0$
    \For{$k = 1:N$}
    \State $\hat{s_k} \gets \text{unit vector in random direction}$
    \State $I \gets I + i(\hat{s_k})$
    \EndFor
    \State $I \gets \frac{s_n}{N} I$
    \State \Return $I$
    \EndProcedure
  \end{algorithmic}
  \label{alg:integral}
\end{algorithm}

\section{Examples and demonstrations}
\label{sec:numerical-results}

\newcommand{\volumewindowlength}{1000}

The NSMC algorithm was used to estimate the volumes of bodies
with various extent densities, and to estimate various other
integrals. The relative error between the estimate and the true value
obtained from a known analytical expression was used as a stopping
criterion. The number of samples on $\sphereset$ required for
\volumewindowlength{} consecutive estimates to achieve a relative
tolerance of 0.05, 0.1 and 0.2, is plotted against the dimension of
the problem in
\cref{fig:volume-uniform,fig:volume-beta,fig:volume-arcsine,fig:integral-polynomial,fig:integral-gaussian,fig:integral-x-coordinate}. A
direct comparison with an implementation \cite{lovasz2015src} of a
simulated annealing MCMC method \cite{lovasz2003simulated} to estimate
the volume of certain convex bodies is shown in
\cref{fig:volume-cube,fig:volume-ellipsoid}.

\def\power10typesetter#1{%
  \pgfkeys{/pgf/number format/.cd,sci,retain unit mantissa=false}%
  \pgfmathprintnumber{#1}%
}
\newcommand{\plot}[5]{
  \subfloat[#2]{
    \begin{tikzpicture}[scale=0.9]
      \datavisualization
      [scientific axes,
      x axis = {attribute=dimension, label=Dimension},
      y axis = {attribute=samples, logarithmic, ticks={minor steps
          between steps=9, tick
          typesetter/.code=\power10typesetter{####1}}, label=Average
        number of samples, include value=1000},
      visualize as line/.list = {tol0.2, tol0.1, tol0.05},
      tol0.2 = {style={mark=x}, #3},
      tol0.1 = {style={mark=x}, #4},
      tol0.05 = {style={mark=x}, #5}]
      data [set=tol0.2, read from file=data/#1-0.2.dat,separator=\space]
      data [set=tol0.1, read from file=data/#1-0.1.dat,separator=\space]
      data [set=tol0.05, read from file=data/#1-0.05.dat,separator=\space];
    \end{tikzpicture}
    \label{fig:#1}
  }
}

\begin{figure}
  \centering
  \plot{volume-uniform}{\uniformdist{0}{1}}
  {label in data={text'={$\text{tol}=0.2$}, when=dimension is 60}}
  {label in data={text'={$\text{tol}=0.1$}, when=dimension is 60}}
  {label in data={text={$\text{tol}=0.05$}, when=dimension is 50}}
  \plot{volume-beta}{\betadist{2}{2}}
  {label in data={text'={$\text{tol}=0.2$}, when=dimension is 30}}
  {label in data={text'={$\text{tol}=0.1$}, when=dimension is 40}}
  {label in data={text={$\text{tol}=0.05$}, when=dimension is 50}}

  \plot{volume-arcsine}{Arcsine}
  {label in data={text={$\text{tol}=0.2$}, when=dimension is 50}}
  {label in data={text'={$\text{tol}=0.1$}}}
  {label in data={text'={$\text{tol}=0.05$}, when=dimension is 60}}
  \plot{integral-gaussian}{Gaussian integrand}
  {label in data={text'={$\text{tol}=0.2$}}}
  {label in data={text'={$\text{tol}=0.1$}, when=dimension is 60}}
  {label in data={text'={$\text{tol}=0.05$}, when=dimension is 60}}

  \plot{integral-polynomial}{Polynomial integrand}
  {label in data={text'={$\text{tol}=0.2$}, when=dimension is 50}}
  {label in data={text'={$\text{tol}=0.1$}}}
  {label in data={text'={$\text{tol}=0.05$}, when=dimension is 60}}
  \plot{integral-x-coordinate}{x-coordinate integrand}
  {label in data={text'={$\text{tol}=0.2$}, when=dimension is 60}}
  {label in data={text'={$\text{tol}=0.1$}}}
  {label in data={text={$\text{tol}=0.05$}, when=dimension is 40}}
  \caption[Scaling of the number of samples with dimension to estimate
  volumes and integrals using extent sampling]{The number of samples
    on $\sphereset$ required for \volumewindowlength{} consecutive
    estimates to achieve a relative tolerance (tol) of 0.05, 0.1 and
    0.2, is plotted against the dimension of the problem.
    \cref{fig:volume-uniform,fig:volume-beta,fig:volume-arcsine} show
    respectively the number of samples required to estimate the volume
    of a body with \uniformdist, \betadist and arcsine
    extent
    densities. \cref{fig:integral-gaussian,fig:integral-polynomial,fig:integral-x-coordinate}
    show respectively the number of samples required to estimate the
    integral of the radially symmetric Gaussian integrand of
    \cref{eq:gaussian-integrand}, the radially symmetric oscillatory
    polynomial integrand of \cref{eq:polynomial-integrand} and the
    radially asymmetric x-coordinate integrand of
    \cref{eq:x-coordinate-integrand}, over domains with
    \uniformdist extent density. Note that this domain is highly
    eccentric and can have an arbitrary geometry. The number of
    samples in
    \cref{fig:volume-uniform,fig:volume-arcsine,fig:volume-beta} was
    averaged over \distributiontrials{} trials, and the number of
    samples in
    \cref{fig:integral-polynomial,fig:integral-gaussian,fig:integral-x-coordinate}
    was averaged over \integraltrials{} trials.}
\end{figure}

\subsection{Estimation of the volume represented by extent densities}
\label{sec:volume-distributions-numerical-results}

For estimation of the volumes of bodies represented by various extent
densities, see
\cref{fig:volume-uniform,fig:volume-beta,fig:volume-arcsine}.

\subsubsection{Uniform extent density}

The estimation of the volume of a body with extents uniformly
distributed between 0 and 1 is shown in \cref{fig:volume-uniform}. The
extent density and true volume of a body with extents uniformly
distributed between $a$ and $b$ are
\begin{equation}
  f_R(r) =
  \begin{cases}
    \frac{1}{b-a} & r \in \intcc{a, b} \\
    0 & \text{otherwise}
  \end{cases}
\end{equation}
\begin{equation}
  V = \frac{v_n}{n+1} \sum_{k=0}^n a^k b^{n-k}
\end{equation}

\subsubsection{Beta extent density}

The estimation of the volume of a body with a \betadist{2}{2}
distribution of extents is shown in \cref{fig:volume-beta}. The
probability density of the general \betadist{}{} distribution and the
true volume of a body with extents distributed as the general
\betadist{}{} distribution are
\begin{equation}
  f_R(r) =
  \begin{cases}
    \frac{x^{\alpha-1}(1-x)^{\beta-1}}{B(\alpha,\beta)} & x \in \intcc{0,1} \\
    0 & otherwise
  \end{cases}
\end{equation}
\begin{equation}
  V = v_n \prod_{i=0}^{k-1} \frac{\alpha + i}{\alpha + \beta + i}
\end{equation}

\subsubsection{Arcsine extent density}

The estimation of the volume of a body with an arcsine distribution of
extents is shown in \cref{fig:volume-arcsine}. The arcsine
distribution is a special case of the beta distribution with
$\alpha = \beta = \frac{1}{2}$.

\subsection{Estimation of the volume of convex bodies}
\label{sec:volume-bodies-numerical-results}

For estimation of the volumes of various convex bodies, see
\cref{fig:volume-cube,fig:volume-ellipsoid}. We compare against an
implementation \cite{lovasz2015src} of the simulated annealing MCMC
method \cite{lovasz2003simulated}. While we compare the number of
samples in each method, note that the cost of each sample in simulated
annealing MCMC is much higher than the cost of each sample in NSMC. Iterating the point of reference to the nominal centre of the body, using $\sim n$ pairs of extents of a body in  directions $\hat{s}$ and $-\hat{s}$ is relatively trivial in NSMC, and for bodies of higher reflection symmetries this convergence is faster. The presented results average over varying origins uniformly distributed in a sphere co-centred with the convex body.

Since the extent density of a convex shape is a function of dimension $n$, the expected samples required by NSMC need not a
monotonic function of $n$.  For example, in the case of cube, there is a reduction of the required number of the extent samples by a factor $\sim 2^n$ due to its symmetry, but the $n^{th}$ moment of the
extent density increases approximately as $(\frac{\sqrt{n}}{2})^n$. This
results in very favorable comparison of the naive NSMC with MCMC up to
moderate values of $n$, and this advantage over MCMC is lost as $n$ becomes
larger than 100 where the extent density becomes tailed.

\newcommand{\mcmccomparisonplot}[4]{
  \subfloat[#2]{
    \begin{tikzpicture}[scale=0.9]
      \datavisualization
      [scientific axes,
      x axis = {attribute=dimension, label=Dimension},
      y axis = {attribute=samples, logarithmic, ticks={minor steps
          between steps=9, tick
          typesetter/.code=\power10typesetter{####1}}, label=Number of
      samples},
      visualize as line/.list = {nsmc, mcmc},
      nsmc = {style={mark=x}, #3},
      mcmc = {style={mark=x}, #4}]
      data [set=nsmc, read from
      file=data/volume-offcenter-#1.dat,separator=\space]
      data [set=mcmc, read from
      file=data/mcmc-#1.dat,separator=\space];
    \end{tikzpicture}
    \label{fig:volume-#1}
  }
}

\begin{figure}
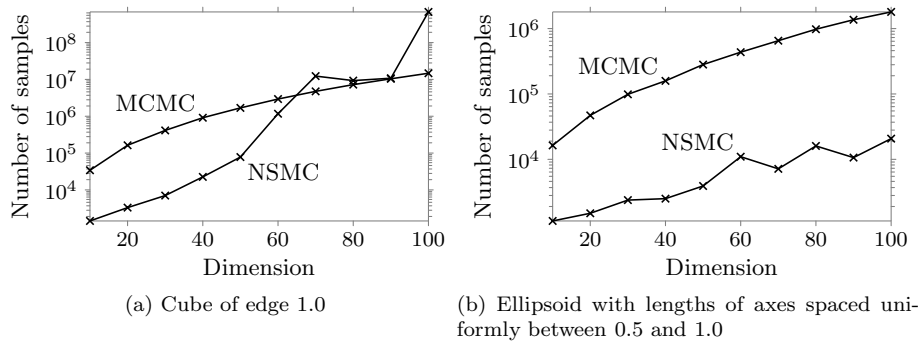

  \centering
  \mcmccomparisonplot{cube}{Cube of edge \cubeedge{}}
  {label in data={text'={NSMC}, when=dimension is 50}}
  {label in data={text={MCMC}, when=dimension is 40}}
  \mcmccomparisonplot{ellipsoid}{Ellipsoid with lengths of axes spaced uniformly
    between \ellipsoidaxesstart{} and \ellipsoidaxesstop{}}
  {label in data={text={NSMC}, when=dimension is 60}}
  {label in data={text={MCMC}, when=dimension is 40}}
  \caption[Scaling of the number of samples with dimension to estimate
  volumes of convex bodies extent sampling]{The number of samples
    required in NSMC and an implementation \cite{lovasz2015src} of
    simulated annealing MCMC \cite{lovasz2003simulated} to estimate
    volume to a relative tolerance of \bodyrtol{}, is plotted against
    the dimension of the problem. For NSMC, the algorithm was stopped
    when \volumewindowlength{} consecutive estimates fell within the
    required relative tolerance of the analytically known true volume.
    \cref{fig:volume-cube,fig:volume-ellipsoid} show respectively the
    number of samples required to estimate the volume of a cube with
    edge \cubeedge{}, and ellipsoid with lengths of axes spaced uniformly between
    \ellipsoidaxesstart{} and \ellipsoidaxesstop{}. NSMC was averaged
    over \offsettrials{} trials, each trial estimating the volume from
    a random center of reference. In \cref{fig:volume-cube}, the
    random center of reference was uniformly distributed in an
    co-centered sphere with a diameter
    \cubeoffsetmaximumpercent{}\% of the edge of the cube. In
    \cref{fig:volume-ellipsoid}, the random center of reference was
    uniformly distributed in an co-centered sphere with a diameter
    \ellipsoidoffsetmaximumpercent{}\% of the longest axis of the
    ellipsoid. While we compare the number of samples in each method,
    note that the cost of each sample in simulated annealing MCMC is
    much higher than the cost of each sample in NSMC.}
\end{figure}

\subsection{Estimation of arbitrary integrals}
\label{sec:integral-estimation-numerical-results}

For estimation of various integrals, see
\cref{fig:integral-polynomial,fig:integral-gaussian,fig:integral-x-coordinate,fig:integral-gaussian}. The
irregular domain chosen has a very large eccentricity with its extents
distributed uniformly between 0 and 1. The Gaussian integrand is in
itself radially symmetric, but note that it is sensitive to any small
errors in sampling such an eccentric domain when $n$ is large. The
next example given by a polynomial includes an additional oscillatory
behavior, but the proposed hybrid approach is robust for such
integrands as well. On the other hand, the example called the
x-coordinate integrand is highly asymmetric radially. The integrals
chosen have exact analytical expressions to confirm convergence for
all dimensions, as they can indeed be reduced to functions of a single
effective variable. Note that in these figures, the number of samples
indicates the number of direction vectors sampled. This does not
include the cost of the deterministic quadrature in evaluating
$i(\hat{s})$ along a direction. The precise cost of this quadrature depends on the
integrand, but note that the cost of evaluating a given scalar function
$h$ increases as \bigo{n} with the number of cardinal directions $n$,
and the number of evaluations of the integrand $\rho^{n-1}h$ required for the quadrature also increase approximately as $n$, making the computing
effort in evaluating $i(\hat{s})$ scale at most as
\bigo{n^2}. The examples demonstrate the \bigo{n} scaling of the number of random
samples required on $\sphereset$ with the dimension $n$ of the non-convex
domain of an arbitrary integral, with the overall computing effort thus
scaling as \bigo{n^3} at most. Some problems of integration where
domains represent a tailed distribution of large extents
with appropriately aligned highly asymmetric integrands, can render the
above approach ineffective. Such special cases require an important
sampling of the partial integral $i(\hat{s})$, and they will be
addressed elsewhere.

\subsubsection{Gaussian integrand}

The estimation of the integral of the following radially symmetric
Gaussian integrand, where $r$ is the radial coordinate, over a domain
with extents distributed uniformly between 0 and 1 is shown in
\cref{fig:integral-gaussian}.
\begin{equation}
  \label{eq:gaussian-integrand}
  h(r) = \exp\del{-\frac{r^2}{2}}
\end{equation}
If $\gamma$ is the lower incomplete gamma function, then the true
partial integral along a direction $\hat{s}$ and the true integral
over a domain with extents uniformly distributed between 0 and $r_0$
are
\begin{equation}
  i(\hat{s}) = 2^{\frac{n}{2} - 1} \gamma\del{\frac{n}{2}, \frac{\cbr{S(\hat{s})}^2}{2}}
\end{equation}
\begin{equation}
  I = \frac{2^{\frac{n-1}{2}}s_n}{r_0}
  \sbr{\frac{r_0}{\sqrt{2}}\gamma\del{\frac{n}{2}, \frac{r_0^2}{2}} - \gamma\del{\frac{n+1}{2}, \frac{r_0^2}{2}}}
\end{equation}

\subsubsection{Polynomial integrand}

The estimation of the integral of the following radially symmetric
oscillatory polynomial integrand, where $r$ is the radial coordinate,
over a domain with extents distributed uniformly between 0 and 1 is
shown in \cref{fig:integral-polynomial}.
\begin{equation}
  \label{eq:polynomial-integrand}
  h(r) = (r - 0.25)(r - 0.50)(r - 0.75) = r^3 - 1.5r^2 + 0.6875r - 0.09375
\end{equation}
For a general polynomial of the form below with $a_k$ as its
coefficients,
\begin{equation}
  h(r) = \sum_{k=0}^{m-1} a_kr^k
\end{equation}
the true partial integral along a direction $\hat{s}$ and the true integral
over a domain with extents uniformly distributed between 0 and $r_0$ are
\begin{equation}
  i(\hat{s}) = \sum_{k=0}^{m-1} \frac{a_k}{n + k} \{S(\hat{s})\}^{n+k}
\end{equation}
\begin{equation}
  I = s_n r_0^n \sum_{k=0}^{m-1} \frac{a_k}{(n+k)(n+k+1)} r_0^k
\end{equation}

\subsubsection{x-coordinate integrand}

The estimation of the integral of the following radially asymmetric
integrand that maps a vector $\vec{x}$ to the absolute value of its
coordinate along the first cardinal direction, over a domain with
extents uniformly distributed between 0 and 1 is shown in
\cref{fig:integral-x-coordinate}. Here, $\hat{x}_1$ is the unit vector
along the first cardinal direction.
\begin{equation}
  \label{eq:x-coordinate-integrand}
  h(\vec{x}) = \envert{\vec{x} \cdot \hat{x}_1}
\end{equation}
The true partial integral along a direction $\hat{s}$ and the true
integral over a domain with extents uniformly distributed between 0
and $r_0$ are
\begin{equation}
  i(\hat{s}) = \frac{\cbr{S(\hat{s})}^{n+1}}{2\pi^2} \frac{s_{n+3}}{s_n}
\end{equation}
\begin{equation}
  I = \frac{s_{n+3}}{2\pi^2(n+2)} r_0^{n+1}
\end{equation}

\appendix
\section{Multi-valued extent function}
\label{sec:multi-valued-extent-function}

In case the extent function $S$ is multi-valued (see
\cref{fig:multi-valued}), the volume of a body, whether it is simply
connected or not, is
\begin{equation}
  V = \oint_{\sphereset} \cbr{\sum_{\text{odd }j}\int_{S^{j-1}(\hat{s})}^{S^j(\hat{s})} \rho^{n-1}\dif \rho} \dif \hat{s}
\end{equation}
with $j = 1, 2, 3, \dotsc$, and the above can again be reduced to the
statistical estimate of the volume as
\begin{equation}
  V = \frac{1}{n} \oint_{\sphereset} r^n \dif \hat{s} =  \cbr{\frac{1}{n} \oint_{\sphereset} \dif \hat{s}} \cbr{\int_0^\infty r^n f_R(r) \dif r}
\end{equation}
with the random extent $R$ now generalized as
\begin{equation}
    R^n = \sum_j (-1)^{j+1} R_j^n
\end{equation}
where $R_1 < R_2 < R_3\dots$ are the random extents representing
multiple values $S^j$ for a given direction $\hat{s}$, and $S^0$=0
always. Note that the largest natural number $j$ representing number
of extents in a given direction, is always odd for a closed body
defined by a bounding surface $S$ around the origin of reference. In
case the origin is outside the closed body, the number of extents is even
valued and this can be treated by a simple negation of signs in the
above equation defining the generalized extents.

\bibliographystyle{siamplain}
\bibliography{references}

\end{document}

%% file: data/volume-body-parameters.tex
\newcommand{\bodyrtol}{0.1}
\newcommand{\cubeedge}{1.0}
\newcommand{\ellipsoidaxesstart}{0.5}
\newcommand{\ellipsoidaxesstop}{1.0}
\newcommand{\cubeoffsetmaximumpercent}{6.25}
\newcommand{\ellipsoidoffsetmaximumpercent}{10}
\newcommand{\offsettrials}{10}

%% file: data/volume-distribution-parameters.tex
\newcommand{\distributiontrials}{100}

%% file: data/integral-parameters.tex
\newcommand{\integraltrials}{1000}

%% file: extent-sampling.bbl
\begin{thebibliography}{10}

\bibitem{allgower1986volume}
{\sc E.~L. Allgower and P.~H. Schmidt}, {\em Computing volumes of polyhedra},
  Mathematics of Computation, 46 (1986), pp.~171--174.

\bibitem{arun2021thesis}
{\sc I.~Arun}, {\em Algorithms for estimating integrals in high-dimensional
  spaces}, doctoral thesis, submitted to Indian Institute of Science (2021).

\bibitem{arun2021cones}
{\sc I.~Arun and M.~Venkatapathi}, {\em An {O}(n) algorithm for generating
  uniform random vectors in n-dimensional cones}, 2021,
  \url{https://arxiv.org/abs/2101.00936}.

\bibitem{chen2016scaling}
{\sc Y.~Chen, D.~Keyes, K.~J. Law, and H.~Ltaief}, {\em Accelerated
  dimension-independent adaptive {M}etropolis}, SIAM Journal on Scientific
  Computing, 38 (2016), pp.~S539--S565.

\bibitem{constantine2016scaling}
{\sc P.~G. Constantine, C.~Kent, and T.~Bui-Thanh}, {\em Accelerating {M}arkov
  chain {M}onte {C}arlo with active subspaces}, SIAM Journal on Scientific
  Computing, 38 (2016), pp.~A2779--A2805.

\bibitem{lovasz2015src}
{\sc B.~Cousins}, {\em Volume and sampling},
  \url{https://in.mathworks.com/matlabcentral/fileexchange/43596-volume-and-sampling}.

\bibitem{dick2007qmc}
{\sc J.~Dick}, {\em Explicit constructions of quasi-{M}onte {C}arlo rules for
  the numerical integration of high-dimensional periodic functions}, SIAM
  Journal on Numerical Analysis, 45 (2007), pp.~2141--2176.

\bibitem{dyer1991computing}
{\sc M.~Dyer and A.~Frieze}, {\em Computing the volume of convex bodies: a case
  where randomness provably helps}, Probabilistic Combinatorics and its
  Applications, 44 (1991), pp.~123--170.

\bibitem{dyer1988complexity}
{\sc M.~E. Dyer and A.~M. Frieze}, {\em On the complexity of computing the
  volume of a polyhedron}, SIAM Journal on Computing, 17 (1988), pp.~967--974.

\bibitem{feng2018scaling}
{\sc Z.~Feng and J.~Li}, {\em An adaptive independence sampler {MCMC} algorithm
  for {B}ayesian inferences of functions}, SIAM Journal on Scientific
  Computing, 40 (2018), pp.~A1301--A1321.

\bibitem{fok1989volume}
{\sc D.~Fok and D.~Crevier}, {\em Volume estimation by {M}onte {C}arlo
  methods}, Journal of Statistical Computation and Simulation, 31 (1989),
  pp.~223--235.

\bibitem{ge2015fast}
{\sc C.~Ge and F.~Ma}, {\em A fast and practical method to estimate volumes of
  convex polytopes}, in International Workshop on Frontiers in Algorithmics,
  Springer, 2015, pp.~52--65.

\bibitem{henrion2009approximate}
{\sc D.~Henrion, J.~B. Lasserre, and C.~Savorgnan}, {\em Approximate volume and
  integration for basic semialgebraic sets}, SIAM Review, 51 (2009),
  pp.~722--743.

\bibitem{jaekel2011monte}
{\sc U.~Jaekel}, {\em A {M}onte {C}arlo method for high-dimensional volume
  estimation and application to polytopes.}, in ICCS, 2011, pp.~1403--1411.

\bibitem{kannan1997random}
{\sc R.~Kannan, L.~Lov{\'a}sz, and M.~Simonovits}, {\em Random walks and an
  o*(n5) volume algorithm for convex bodies}, Random Structures \& Algorithms,
  11 (1997), pp.~1--50.

\bibitem{keshavarzzadeh2018numerical}
{\sc V.~Keshavarzzadeh, R.~M. Kirby, and A.~Narayan}, {\em Numerical
  integration in multiple dimensions with designed quadrature}, SIAM Journal on
  Scientific Computing, 40 (2018), pp.~A2033--A2061.

\bibitem{lawrence1991volume}
{\sc J.~Lawrence}, {\em Polytope volume computation}, Mathematics of
  Computation, 57 (1991), pp.~259--271.

\bibitem{lovasz2003simulated}
{\sc L.~Lov{\'a}sz and S.~Vempala}, {\em Simulated annealing in convex bodies
  and an o*(n/sup 4/) volume algorithm}, in Foundations of Computer Science,
  2003. Proceedings. 44th Annual IEEE Symposium on, IEEE, 2003, pp.~650--659.

\bibitem{simonovits2003compute}
{\sc M.~Simonovits}, {\em How to compute the volume in high dimension?},
  Mathematical Programming, 97 (2003), pp.~337--374.

\bibitem{speevak1986volume}
{\sc T.~Speevak}, {\em An efficient algorithm for obtaining the volume of a
  special kind of pyramid and application to convex polyhedra}, Mathematics of
  Computation, 46 (1986), pp.~531--536.

\bibitem{vollmer2015scaling}
{\sc S.~J. Vollmer}, {\em Dimension-independent {MCMC} sampling for inverse
  problems with non-{G}aussian priors}, SIAM/ASA Journal on Uncertainty
  Quantification, 3 (2015), pp.~535--561.

\end{thebibliography}
